\documentclass[12pt,twoside]{amsart}

\usepackage{amsmath,amsthm,amscd,amssymb,mathrsfs,graphicx,amsfonts,mathrsfs}
\usepackage{amsfonts}
\usepackage{amssymb,enumerate}
\usepackage{amsthm}
\usepackage[all]{xy}
\usepackage{hyperref}
\allowdisplaybreaks[4] \footskip=12pt
\renewcommand{\uppercasenonmath}[1]{}

\numberwithin{equation}{section} \theoremstyle{plain}
\newtheorem*{thm*}{Main Theorem}
\newtheorem{thm}{Theorem}[section]

\newtheorem*{cor*}{Corollary}

\newtheorem*{lem*}{Lemma}

\newtheorem*{fact*}{Fact}

\newtheorem*{nota*}{Notation}
\newtheorem{prop}[thm]{Proposition}
\newtheorem*{prop*}{Proposition}
\newtheorem{rem}[thm]{Remark}
\newtheorem*{rem*}{Remark}

\newtheorem*{observation*}{Observation}
\newtheorem{exa}[thm]{Example}
\newtheorem*{exa*}{Example}
\newtheorem{df}[thm]{Definition}
\newtheorem*{df*}{Definition}

\newtheorem*{conj*}{Conjecture}

\newcommand{\Ker}{\operatorname{Ker}}

\renewcommand{\geq}{\geqslant}

\renewcommand{\ker}{\Ker}
\begin{document}
\footnote[0]{}
\begin{center}
{\large  \bf  Model structures and recollements induced by duality pairs}

\vspace{0.5cm}  Wenjing Chen, Ling Li, Yanping Rao\\
Department of Mathematics, Northwest Normal University, Lanzhou 730070,
China\\
E-mails: chenwj@nwnu.edu.cn, liling980323@163.com, raoyanping0806@163.com
\end{center}

\bigskip
\centerline { \textsc{Abstract}}
\leftskip10truemm \rightskip10truemm \noindent We give some equivalent characterizations of $\mathcal{GP}$, the class of Gorenstein $(\mathcal{L}, \mathcal{A})$-projective modules, and construct some model structures associated to duality pairs and Frobenius pairs. Moreover, some rings are described by Frobenius pairs. Meanwhile, we investigate strongly Gorenstein $(\mathcal{L}, \mathcal{A})$-projective modules and obtain some equivalent characterizations of them. Also, some model structures and recollements associated to strongly Gorenstein $(\mathcal{L}, \mathcal{A})$-projective modules are constructed.

\leftskip0truemm \rightskip0truemm
\bigskip

\section{Introduction}
The notion of a duality pair of $R$-modules was introduced by Holm and J$\mathrm{\phi}$rgensen in [15]. Duality pairs exist extensively (refer to [13, 15, 18]). Indeed, duality pairs are related to purity and the existence of covers and envelopes, which makes duality pairs be very useful in relative homological algebra and attract wide attention. Gillespie investigated Gorenstein homological algebra with respect to a duality pair and constructed some relevant model structures in [13]. Complete duality pairs play an important role in constructing model structures. Motivated by this, we continue to study Gorenstein homological algebra with respect to a duality pair and model structures associated to duality pairs.

Gillespie introduced exact model structures in the exact category in [10], and gave a correspondence between the exact model structure and two complete cotorsion pairs, called Hovey-Gillespie correspondence in [1]. Based on this fact, Becerril and coauthors showed how to construct an exact model structure from a Frobenius pair in [1], which tells us that Frobenius pairs also have an important effect on constructing model structures. As a result, we have tried to find Frobenius pairs by Gorenstein objects with respect to cotorsion pairs in [5], and obtained some meaningful conclusions. Wang and coauthors introduced and studied Gorenstein flat modules with respect to duality pairs, which enriches Gorenstein homological algebra with respect to a duality pair in [18].
However, it is well known that a perfect duality pair
can induce a perfect cotorsion pair.
This builds a bridge between duality pairs and cotorsion pairs. So we got some model structures associated to duality pairs and Frobenius pairs by applying some known results in [5].
In this paper, we give some characterizations of rings by Frobenius pairs.

The recollement of triangulated categories was introduced by Beilinson, Bernstein and Deligne in a geometric setting in [2], which plays an important role in algebraic geometry and in representation theory. Gillespie described a general correspondence between projective (injective) recollements of triangulated categories and projective (injective) cotorsion pairs in [12]. This provides a model category description of these recollement situations.
In this paper, we investigate special Gorenstein $(\mathcal{L}, \mathcal{A})$-projective modules, called strongly Gorenstein $(\mathcal{L}, \mathcal{A})$-projective modules. Note that for a strongly Gorenstein $(\mathcal{L}, \mathcal{A})$-projective module $G$, $(^{\perp}(G^{\perp}),G^{\perp})$ is a projective cotorsion pair, cogenerated by a set.
Based on this fact, we construct some recollements.
These recollements involve complexes built from $G^{\perp}$ or $^{\perp}(G^{\perp})$.
Naturally, when $(\mathcal{L}, \mathcal{A})$ is the level duality pair, we can get some specific examples.

This paper is organized as follows. In section 2, we give some basic notions. In section 3, we give some equivalent characterizations of $\mathcal{GP}$ and construct some model structures associated to duality pairs and Frobenius pairs. Moreover, we describe some rings by Frobenius pairs. In section 4, we investigate strongly Gorenstein $(\mathcal{L}, \mathcal{A})$-projective modules and obtain some equivalent characterizations of them. Also, some model structures and recollements associated to strongly Gorenstein $(\mathcal{L}, \mathcal{A})$-projective modules are constructed.

\section{Preliminaries}
Next we recall some notions and basic facts which we need in the later sections.

{\bf Frobenius pair.} Let $\mathcal{A}$ be an abelian category, $\mathcal{X},~\mathcal{Y}\subseteq\mathcal{A}$ two classes of objects of $\mathcal{A}$ which can be also regarded as full subcategories of $\mathcal{A}$, and $M$, $N$ objects in $\mathcal{A}$. The relative projective dimension of $M$ with respect to $\mathcal{X}$ is defined as pd$_{\mathcal{X}}(M)=\mathrm{min}\{n\geq0\mid\mathrm{Ext}^{j}_{\mathcal{A}}(M,\mathcal{X})=0~\mathrm{for}~\mathrm{every}~j>n\}$. The relative injective dimension of $N$ with respect to $\mathcal{Y}$ is defined as id$_{\mathcal{Y}}(N)=\mathrm{min}\{n\geq0\mid\mathrm{Ext}^{j}_{\mathcal{A}}(\mathcal{Y},N)=0~\mathrm{for}~\mathrm{every}~j>n\}$. Furthermore, we set pd$_{\mathcal{X}}\mathcal{Y}=\mathrm{sup}\{\mathrm{pd}_{\mathcal{X}}(Y)\mid Y\in\mathcal{Y}\}$ and id$_{\mathcal{X}}\mathcal{Y}=\mathrm{sup}\{\mathrm{id}_{\mathcal{X}}(Y)\mid Y\in\mathcal{Y}\}$.

The class $\mathcal{X}$ is left thick if it is closed under direct summands, extensions and kernels of epimorphisms in $\mathcal{A}$. $\mathcal{X}$ is thick if it is left thick and closed under cokernels of monomorphisms in $\mathcal{A}$. Let $(\mathcal{X},\omega)$ be a pair of classes of objects in $\mathcal{A}$. It is said that $\omega$ is $\mathcal{X}$-injective if id$_{\mathcal{X}}\omega=0$. $\omega$ is called a relative cogenerator in $\mathcal{X}$ if $\omega\subseteq\mathcal{X}$ and for any $X\in\mathcal{X}$, there exists a short exact sequence $0\rightarrow X\rightarrow W\rightarrow X^{'}\rightarrow 0$ with $W\in\omega$ and $X^{'}\in\mathcal{X}$. Definitions of $\mathcal{X}$-projective and a relative generator in $\mathcal{X}$ are dual. $(\mathcal{X},\omega)$ is called a left Frobenius pair if $\mathcal{X}$ is a left thick class, $\omega$ is an $\mathcal{X}$-injective relative cogenerator in $\mathcal{X}$ and $\omega$ is closed under direct summands in $\mathcal{A}$. A left Frobenius pair $(\mathcal{X},\omega)$ is strong if $\omega$ is an $\mathcal{X}$-projective relative generator in $\mathcal{X}$.

The $\mathcal{X}$-resolution dimension of $M$ is the smallest non-negative integer $n$ such that there is an exact sequence $0\rightarrow X_{n}\rightarrow \cdots\rightarrow X_{0}\rightarrow M\rightarrow 0$ with each $X_{i}\in\mathcal{X}$. If such $n$ doesn't exist, we say that the $\mathcal{X}$-resolution dimension of $M$ is infinite. We denote by $\mathcal{X}^{\wedge}$ the class of objects in $\mathcal{A}$ having a finite $\mathcal{X}$-resolution dimension.

Let $(\mathcal{X},\mathcal{Y})$ be a pair of classes of objects in $\mathcal{A}$ and $\omega=\mathcal{X}\cap\mathcal{Y}$. We say that $(\mathcal{X},\mathcal{Y})$ is a left Auslander-Buchweitz-context (left AB-context for short) if the pair $(\mathcal{X},\omega)$ is a left Frobenius pair, $\mathcal{Y}$ is thick and $\mathcal{Y}\subseteq\mathcal{X}^{\wedge}$.

{\bf Exact category.} An exact category is a pair $(\mathcal{B}, \varepsilon)$ consisting of an additive category $\mathcal{B}$ and an exact structure $\varepsilon$ on $\mathcal{B}$. Elements of $\varepsilon$ are called short exact sequences.

An exact category $(\mathcal{B}, \varepsilon)$ is a Frobenius category if $(\mathcal{B}, \varepsilon)$ has enough projectives and enough injectives such that the projectives coincide with the injectives. For any objects $M,~N\in \mathcal{B}$, let $\mathcal{P}(M,N)$ denote the abelian group of morphisms from $M$ to $N$ factoring through some projective object. Furthermore, the stable category of $\mathcal{B}$ denotes $\underline{\mathcal{B}}:= \mathcal{B}/\mathcal{P}$, where the objects of $\underline{\mathcal{B}}$ are same as the objects of $\mathcal{B}$ and $\mathrm{Hom}_{\underline{\mathcal{B}}}(M,N):= \mathrm{Hom}_{\mathcal{B}}(M,N)/\mathcal{P}(M,N)$. It is well known that $\underline{\mathcal{B}}$ is a triangulated category.

{\bf Recollement.} Let $\mathcal{T}^{\prime}$, $\mathcal{T}$ and $\mathcal{T}^{\prime\prime}$ be triangulated categories. A recollement of $\mathcal{T}$ relative to $\mathcal{T}^{\prime}$ and $\mathcal{T}^{\prime\prime}$ is a diagram of triangulated functors
$$\xymatrix{\mathcal{T}^{\prime}\ar^-{i_{\ast}}[r]&\mathcal{T}\ar^-{j^{\ast}}[r]
\ar^-{i^{!}}@/^1.2pc/[l]\ar_-{i^{\ast}}@/_1.6pc/[l]
&\mathcal{T}^{\prime\prime}\ar^-{j_{\ast}}@/^1.2pc/[l]\ar_-{j_{!}}@/_1.6pc/[l]}$$
satisfying the following conditions:

(R1) $(i^{\ast},i_{\ast},i^{!})$ and $(j_{!},j^{\ast},j_{\ast})$ are adjoint triples,

(R2) $j^{\ast}i_{\ast}=0$,

(R3) $i_{\ast}$, $j_{!}$ and $j_{\ast}$ are full embeddings,

(R4) any object $X$ in $\mathcal{T}$ determines distinguished triangles
$i_{\ast}i^{!}X\rightarrow X\rightarrow j_{\ast}j^{\ast}X\rightarrow (i_{\ast}i^{!}X)[1]$
and $j_{!}j^{\ast}X\rightarrow X\rightarrow i_{\ast}i^{\ast}X\rightarrow (j_{!}j^{\ast}X)[1]$ (see [2]).

{\bf Cotorsion pair.} Let $\mathcal{A}$ be an abelian category. A cotorsion pair is a pair $(\mathcal{X}, \mathcal{Y})$ of classes of objects
in $\mathcal{A}$ such that $\mathcal{X}^\perp=\mathcal{Y}$ and $\mathcal{X}=$
$^\perp\mathcal{Y}$, where $\mathcal{X}^\perp=\{A\in\mathcal{A}\hspace{0.03cm}|\hspace{0.03cm}\mathrm{Ext}^{1}_{\mathcal{A}}(X,A)=0,\ \forall\ X\in\mathcal{X}\}$ and $^\perp\mathcal{Y}=\{B\in\mathcal{A}\hspace{0.03cm}|\hspace{0.03cm}\mathrm{Ext}^{1}_{\mathcal{A}}(B,Y)=0,\ \forall\ Y\in\mathcal{Y}\}$.
A cotorsion pair $(\mathcal{X}, \mathcal{Y})$ is said to be complete if it has enough projectives and injectives, i.e.,
for any object $A\in\mathcal{A}$, there are exact
sequences $0\rightarrow Y\rightarrow X\rightarrow A\rightarrow 0$ and $0\rightarrow A\rightarrow Y'\rightarrow X'\rightarrow 0$ respectively with $Y,~Y'\in\mathcal{Y}$ and $X,~X'\in\mathcal{X}$.
A cotorsion pair $(\mathcal{X}, \mathcal{Y})$ is said to be hereditary if $\mathrm{Ext}^i_\mathcal{A}(X,Y)=0$ for all
$X\in\mathcal{X},Y\in\mathcal{Y}$ and all $i\geq 1$. It follows from [12, Lemma 2.3] that for a hereditary cotorsion pair $(\mathcal{X}, \mathcal{Y})$, $\mathcal{X}$ is closed under kernels of epimorphisms and $\mathcal{Y}$ is closed under cokernels of monomorphisms. In addition, let $\mathcal{A}$ have enough projectives. We call a complete cotorsion pair $(\mathcal{X}, \mathcal{Y})$ a projective cotorsion pair if $\mathcal{Y}$ is thick and $\mathcal{X}\cap\mathcal{Y}$ coincides with the class of projective objects.

Unless stated to the contrary, we assume in the following that $R$ is an associative ring with an identity, and all modules are left $R$-modules. $R$-Mod denotes the category of all left $R$-modules. $\mathcal{P}$ denotes the class of all projective
left $R$-modules. $\mathcal{F}$ denotes the class of all flat left $R$-modules.
For an $R$-module $M$, $M^{+}=\mathrm{Hom}_{\mathbb{Z}}(M,\mathbb{Q/Z})$ denotes the character module of $M$.
For some unexplained results, we refer the reader to [4, 7, 8, 14, 16, 17].

{\bf Complex.} A complex is a sequence of $R$-modules
$$C=\cdots \rightarrow C_{2}\stackrel{d_{2}}\rightarrow C_{1}\stackrel{d_{1}}\rightarrow C_{0}\stackrel{d_{0}}\rightarrow C_{-1}\stackrel{d_{-1}}\rightarrow C_{-2}\stackrel{}\rightarrow\cdots$$
together with homomorphisms such that $d_{n}d_{n+1}=0$ for all $n\in \mathbb{Z}$.
The $n$th cycle ($n$th boundary) of $C$ is defined as Ker$d_{n}$ (Im$d_{n+1}$) and is denoted by $Z_{n}(C)~(B_{n}(C))$. $C$ is called exact or acyclic if Ker$d_{n}=~$Im$d_{n+1}$ for each $n\in \mathbb{Z}$. We use Ch$(R)$ to denote the category of complexes of $R$-modules.

{\bf Duality pair.} A duality pair over $R$ is a pair $(\mathcal{L}, \mathcal{A})$, of classes of $R$-modules, satisfying $L\in\mathcal{L}$ if and only if $L^{+}\in\mathcal{A}$, and $\mathcal{A}$ is closed under direct summands and finite direct sums, where $\mathcal{A}$ is a class of right $R$-modules. A duality pair
$(\mathcal{L}, \mathcal{A})$ is called perfect if $\mathcal{L}$ contains the module $_RR$, and is closed under coproducts and extensions. $\{\mathcal{L}, \mathcal{A}\}$ is a symmetric duality pair over $R$ if $(\mathcal{L}, \mathcal{A})$ and $(\mathcal{A}, \mathcal{L})$ are duality pairs. A duality pair $(\mathcal{L}, \mathcal{A})$ is complete if $\{\mathcal{L}, \mathcal{A}\}$ is a symmetric duality pair and $(\mathcal{L}, \mathcal{A})$ is a perfect duality pair over $R$.

\section{Gorenstein homological algebra relative to a duality pair and model structures}
The goal of this section is to investigate Gorenstein homological algebra with respect to a fixed complete duality pair $(\mathcal{L}, \mathcal{A})$ and construct some model structures.

An $R$-module $M$ is called Gorenstein $(\mathcal{L}, \mathcal{A})$-projective if there exists a $\mathrm{Hom}_{R}(-,\mathcal{L})$-exact exact sequence $$\mathbb{P}:\cdots\rightarrow P_{1}\rightarrow P_{0}\rightarrow P_{-1}\rightarrow P_{-2}\rightarrow \cdots$$
with each $P_{i}\in\mathcal{P}$ such that $M\cong$~Ker$(P_{-1}\rightarrow P_{-2})$. $\mathcal{GP}$ denotes the class of all Gorenstein $(\mathcal{L}, \mathcal{A})$-projective modules (see [13]). It is well-known that by symmetry, all the kernels, the images and all the cokernels of $\mathbb{P}$ are in $\mathcal{GP}$. $\mathcal{P}\subseteq\mathcal{GP}$ via the exact sequence $0\rightarrow P\stackrel{1}\rightarrow P\rightarrow 0$ for any $P\in\mathcal{P}$ and $\mathrm{Ext}^{i}_{R}(M,L)=0$ for any $L\in\mathcal{L}$ and $i\geq 1$.
$\mathcal{GP}$ is closed under direct sums, extensions, direct summands and kernels of epimorphisms.
Since the perfect duality pair $(\mathcal{L}, \mathcal{A})$ implies $\mathcal{P}\subseteq\mathcal{F}\subseteq\mathcal{L}$ by [13, Proposition 2.3], a Gorenstein $(\mathcal{L}, \mathcal{A})$-projective module is Ding projective, of course, it is Gorenstein projective (see [6, 9]). We use $\mathcal{DP}$ and $\mathcal{GP}(R)$ to denote the class of all Ding projective modules and the class of all Gorenstein projective modules respectively.

The following result gives some equivalent characterizations of Gorenstein $(\mathcal{L}, \mathcal{A})$-projective modules.

\begin{prop}\label{prop:2.4}{\it{The following conditions are equivalent for an $R$-module $M$.

(1) $M$ is a Gorenstein $(\mathcal{L}, \mathcal{A})$-projective module.

(2) There exists a $\mathrm{Hom}_{R}(-,\mathcal{L})$-exact exact sequence $$0\rightarrow M\rightarrow P_{-1}\rightarrow P_{-2}\rightarrow P_{-3}\rightarrow \cdots$$
with each $P_{i}\in\mathcal{P}$ and $\mathrm{Ext}^{i}_{R}(M,L)=0$ for any $L\in\mathcal{L}$ and $i\geq 1$.

(3) There exists an exact sequence $0\rightarrow M\rightarrow P\rightarrow G\rightarrow 0$ with $P\in\mathcal{P}$ and $G\in\mathcal{GP}$.

(4) There exists an exact sequence $0\rightarrow G_{1}\rightarrow G_{0}\rightarrow M\rightarrow 0$ with $G_{1},~G_{0}\in\mathcal{GP}$ and $\mathrm{Ext}^{1}_{R}(M,L)=0$ for any $L\in\mathcal{L}$.

(5) There exists a $\mathrm{Hom}_{R}(-,\mathcal{L})$-exact exact sequence $$0\rightarrow M\rightarrow G_{-1}\rightarrow G_{-2}\rightarrow G_{-3}\rightarrow \cdots$$
with each $G_{i}\in\mathcal{GP}$ and $\mathrm{Ext}^{i}_{R}(M,L)=0$ for any $L\in\mathcal{L}$ and $i\geq 1$.

(6) There exists a $\mathrm{Hom}_{R}(-,\mathcal{L})$-exact exact sequence $$\mathbb{P}:\cdots\rightarrow G_{1}\rightarrow G_{0}\rightarrow G_{-1}\rightarrow G_{-2}\rightarrow \cdots$$
with each $G_{i}\in\mathcal{GP}$ such that $M\cong$~Ker$(G_{-1}\rightarrow G_{-2})$.

(7) There exists some class of $R$-modules $\mathcal{U}$ with $\mathcal{P}\subseteq\mathcal{U}\subseteq\mathcal{GP}$ and there exists a $\mathrm{Hom}_{R}(-,\mathcal{L})$-exact exact sequence $$0\rightarrow M\rightarrow U_{-1}\rightarrow U_{-2}\rightarrow U_{-3}\rightarrow \cdots$$
with each $U_{i}\in\mathcal{U}$ and $\mathrm{Ext}^{i}_{R}(M,L)=0$ for any $L\in\mathcal{L}$ and $i\geq 1$.

(8) There exists some class of $R$-modules $\mathcal{U}$ with $\mathcal{P}\subseteq\mathcal{U}\subseteq\mathcal{GP}$ and there exists a $\mathrm{Hom}_{R}(-,\mathcal{L})$-exact exact sequence$$\mathbb{P}:\cdots\rightarrow U_{1}\rightarrow U_{0}\rightarrow U_{-1}\rightarrow U_{-2}\rightarrow \cdots$$
with each $U_{i}\in\mathcal{U}$ such that $M\cong$~Ker$(U_{-1}\rightarrow U_{-2})$.
}}
\end{prop}
\begin{proof} (1)$\Leftrightarrow$ (2) $\Leftrightarrow$ (3)$\Leftrightarrow$ (5), (1) $\Rightarrow$ (4), (1) $\Rightarrow$ (6) $\Rightarrow$ (5), (2) $\Rightarrow$ (7) $\Rightarrow$ (5), (1) $\Rightarrow$ (8) $\Rightarrow$ (6) are straightforward.

(4) $\Rightarrow$ (1) Since $G_{1}\in\mathcal{GP}$, there exists an exact sequence $0\rightarrow G_{1}\rightarrow P\rightarrow G\rightarrow 0$ with $P\in\mathcal{P}$ and $G\in\mathcal{GP}$. Consider the following pushout diagram$$\xymatrix{
                                                                 & 0 \ar[d]_{}                  & 0  \ar[d]_{}          &    &  \\
 0  \ar[r]^{}       &G_{1}     \ar[d]_{} \ar[r]^{}     &G_{0}  \ar[d]_{}  \ar[r]^{ } & M \ar@{=}[d]^{} \ar[r]^{ } &0  \\
 0 \ar[r]^{}        &P \ar[d]^{} \ar[r]^{}  &X \ar[d]^{ } \ar[r]^{}   &M  \ar[r]^{}  &0  \\
                                                       & G \ar[d]^{} \ar@{=}[r]        &G \ar[d]^{ } &   & \\
                                                        & 0                            &0      &  & }$$
Since $G_{0},~G\in\mathcal{GP}$ and $\mathcal{GP}$ is closed under extensions, $X\in\mathcal{GP}$. Since $\mathrm{Ext}^{1}_{R}(M,P)=0$, $0\rightarrow P\rightarrow X\rightarrow M\rightarrow 0$ is split. Thus $M$ is a direct summand of $X$. So $M$ is a Gorenstein $(\mathcal{L}, \mathcal{A})$-projective module.
\end{proof}

Denote $^{\perp\infty}\mathcal{L}=\{X\in R$-Mod$\hspace{0.03cm}|\hspace{0.03cm}\mathrm{Ext}^{i}_{R}(X,L)=0,\ \forall\ L\in\mathcal{L} ~\mathrm{and}~\ \forall\ i\geq 1\}$. It is easy to see that $^{\perp\infty}\mathcal{L}$ is the class of $\mathcal{L}$-projective modules.

\begin{prop}\label{prop:2.4}{\it{Assume that $\mathcal{L}$ is closed under kernels of epimorphisms. Then $\mathcal{GP}=\mathcal{DP}\cap{^{\perp\infty}\mathcal{L}}$.
}}
\end{prop}
\begin{proof} It is easy to see that $\mathcal{GP}\subseteq\mathcal{DP}\cap{^{\perp\infty}\mathcal{L}}$.
Next, let $M\in\mathcal{DP}\cap{^{\perp\infty}\mathcal{L}}$. Since $M\in\mathcal{DP}$, there is an exact sequence $0\rightarrow M\rightarrow P_0\rightarrow M_0\rightarrow 0$ with $P_0\in\mathcal{P}$ and $M_0\in\mathcal{DP}$. Since $M\in{^{\perp\infty}\mathcal{L}}$, $\mathrm{Ext}^{i}_{R}(M,L)=0$ for any $L\in\mathcal{L}$ and $i\geq 1$. One gets $\mathrm{Ext}^{i}_{R}(M_0,L)=0$ for any $L\in\mathcal{L}$ and for any $i\geq 2$.
We have an exact sequence $0\rightarrow L_0\rightarrow P\rightarrow L\rightarrow 0$ with $P\in\mathcal{P}$. By assumption, $L_0\in\mathcal{L}$. We get an exact sequence
$0=\mathrm{Ext}^{1}_{R}(M_0,P)\rightarrow \mathrm{Ext}^{1}_{R}(M_0,L)\rightarrow \mathrm{Ext}^{2}_{R}(M_0,L_0)=0$. Thus $\mathrm{Ext}^{1}_{R}(M_0,L)=0$. Note that
$M_0\in\mathcal{DP}\cap{^{\perp\infty}\mathcal{L}}$. Continuing this process, we get a $\mathrm{Hom}_{R}(-,\mathcal{L})$-exact exact sequence $0\rightarrow M\rightarrow P_{0}\rightarrow P_{-1}\rightarrow P_{-2}\rightarrow\cdots$ with $P_i\in\mathcal{P}$. By Proposition 3.1, $M$ is a Gorenstein $(\mathcal{L}, \mathcal{A})$-projective module.
Therefore, $\mathcal{GP}=\mathcal{DP}\cap{^{\perp\infty}\mathcal{L}}$.
\end{proof}

\begin{rem}\label{prop:2.4}{\rm{Suppose that $\mathcal{L}$ is closed under kernels of epimorphisms. With the similar method in Proposition 3.2, one can get that $\mathcal{GP}=\mathcal{GP}(R)\cap{^{\perp\infty}\mathcal{L}}$.
}}
\end{rem}

\begin{prop}\label{prop:2.4}{\it{Let $M$ be a Gorenstein $(\mathcal{L}, \mathcal{A})$-projective module. Then the following conditions are equivalent.

(1) $M$ is projective.

(2) $M$ is of finite projective dimension.

(3) $M$ is flat.

(4) $M$ is of finite flat dimension.

(5) $M$ is in $\mathcal{L}$.

(6) $M$ has a finite $\mathcal{L}$-resolution dimension.
}}
\end{prop}
\begin{proof} (1) $\Rightarrow$ (2), (1) $\Rightarrow$ (3) $\Rightarrow$ (5), (3) $\Rightarrow$ (4) and (5) $\Rightarrow$ (6) are obvious.

(2) $\Rightarrow$ (1) Assume $pd_{R}M=n<\infty$. Since $M$ is a Gorenstein $(\mathcal{L}, \mathcal{A})$-projective module, there exists an exact sequence $$0\rightarrow G_{n} \rightarrow P_{n-1}\stackrel{d_{n-1}}\rightarrow P_{n-2}\stackrel{d_{n-2}}\rightarrow \cdots\rightarrow P_{2} \stackrel{d_{2}}\rightarrow P_{1}\stackrel{d_{1}}\rightarrow P_{0}\stackrel{d_{0}}\rightarrow M\rightarrow 0 $$ with each $P_{i}\in\mathcal{P}$ and each $\ker d_{i}$ of this sequence in $\mathcal{GP}$. Then $G_{n}$ is projective. Put $G_{n-1}=\ker d_{n-2}$. Note that $0\rightarrow G_{n}\rightarrow P_{n-1}\rightarrow G_{n-1}\rightarrow 0$ is split. Then $G_{n-1}$ is projective. Repeating this process, we get that $M$ is projective.

(4) $\Rightarrow$ (1) It can be immediately given by the similar method above.

(6) $\Rightarrow$ (5) Let the $\mathcal{L}$-resolution dimension of $M$ be $n$. Then there exists an exact sequence $$0\rightarrow L_{n} \rightarrow L_{n-1}\rightarrow L_{n-2}\rightarrow \cdots\rightarrow L_{2} \rightarrow L_{1}\rightarrow L_{0}\rightarrow M\rightarrow 0 $$ with each $L_{i}\in\mathcal{L}$. By dimension shift, we have $\mathrm{Ext}^{i}_{R}(M,K_{n-1})=0$ for all $i\geq 1$, where $K_{n-1}=$~Im$(L_{n-1}\rightarrow L_{n-2})$, since $L_{n},~L_{n-1}\in\mathcal{L}$ and $M\in\mathcal{GP}$. By continuing this process, $\mathrm{Ext}^{1}_{R}(M,K_{1})=0$ where $K_{1}=$~Im$(L_{1}\rightarrow L_{0})$. Then $0\rightarrow K_{1}\rightarrow L_{0}\rightarrow M\rightarrow 0$ splits. Since $(\mathcal{L}, \mathcal{A})$ is a complete duality pair, $\mathcal{L}$ is closed under direct summands. Then $M\in\mathcal{L}$.

At last, we prove (5) $\Rightarrow$ (1). Since $M$ is a Gorenstein $(\mathcal{L}, \mathcal{A})$-projective module, there exists a short exact sequence $0\rightarrow M\rightarrow P\rightarrow G\rightarrow 0$ with $P\in\mathcal{P}$ and $G\in\mathcal{GP}$ by Proposition 3.1. By assumption, $M\in\mathcal{L}$, $\mathrm{Ext}^{1}_{R}(G,M)=0$, and hence $0\rightarrow M\rightarrow P\rightarrow G\rightarrow 0$ splits.
So $M$ is projective.
\end{proof}

The above proposition shows that $\mathcal{GP}\cap \mathcal{P^{\wedge}}=\mathcal{GP}\cap \mathcal{F^{\wedge}}=\mathcal{GP}\cap \mathcal{L^{\wedge}}=\mathcal{GP}\cap \mathcal{L}=\mathcal{GP}\cap \mathcal{F}=\mathcal{P}$, which is useful in the sequel.

\begin{rem}\label{prop:2.4}{\rm{If $R$ has finite weak global dimension, then $M$ is Gorenstein $(\mathcal{L}, \mathcal{A})$-projective if and only if $M$ is projective. For instance, a Gorenstein $(\mathcal{L}, \mathcal{A})$-projective module is projective over the von Neumann regular ring, where every module is flat.
In particular, if $R$ has finite global dimension, then $M$ is Gorenstein $(\mathcal{L}, \mathcal{A})$-projective if and only if $M$ is projective.
}}
\end{rem}

Since $\mathcal{GP}$ is closed under extensions, it follows from [17, 4.1] that $(\mathcal{GP},\mathcal{E})$ is an exact category, where $\mathcal{E}$ denotes the class of all exact sequences of the form $0\rightarrow L\rightarrow M\rightarrow N\rightarrow 0$ with all terms in $\mathcal{GP}$.
If readers want to know more details about exact categories, please refer to [4, 10].
Next, we will devote to showing that $\mathcal{GP}^{\wedge}$ is closed under extensions. Further, $(\mathcal{GP}^{\wedge},\mathcal{E})$ is an exact category.

\begin{prop}\label{prop:2.4}{\it{Let $n$ be the $\mathcal{GP}$-resolution dimension of $M$. Then for each exact sequence $0\rightarrow K_{n}\rightarrow P_{n-1}\rightarrow \cdots\rightarrow P_{1}\rightarrow P_{0}\stackrel{d_{0}}\rightarrow M\rightarrow 0$ with each $P_{i}\in\mathcal{P}$, $K_{n}$ is Gorenstein $(\mathcal{L}, \mathcal{A})$-projective.
}}
\end{prop}
\begin{proof} By assumption, there exists an exact sequence
$0\rightarrow G_{n}\rightarrow G_{n-1}\rightarrow \cdots\rightarrow G_{1}\stackrel{g_{1}}\rightarrow G_{0}\stackrel{g_{0}}\rightarrow M\rightarrow 0$ with each $G_{i}\in\mathcal{GP}$. For exact sequence $0\rightarrow K_{n}\rightarrow P_{n-1}\rightarrow \cdots\rightarrow P_{1}\rightarrow P_{0}\stackrel{d_{0}}\rightarrow M\rightarrow 0$, Since $P_{i}\in\mathcal{P}$, there is a commutative diagram
$$\xymatrix{
   & 0  \ar[r]^{}& K_{n}\ar[d]^{f_{n}}\ar[r]^{} & P_{n-1} \ar[d]^{f_{n-1}} \ar[r]^{} & \cdots  \ar[r]^{} &P_{1} \ar[d]^{f_{1}} \ar[r]^{} & P_{0}\ar[d]^{f_{0}} \ar[r]^{d_{0}} & M \ar@{=}[d]_{} \ar[r]^{} & 0 \\
  & 0  \ar[r]^{}& G_{n}\ar[r]^{} & G_{n-1} \ar[r]^{}& \cdots  \ar[r]^{}& G_{1} \ar[r]^{g_{1}} & G_{0} \ar[r]^{g_{0}} & M \ar[r]^{} & 0 }$$
such that the mapping cone is exact. We have the following commutative diagram
$$\xymatrix{
    & 0  \ar[r]^{}& 0\ar[d]^{}\ar[r]^{} & 0 \ar[d]^{} \ar[r]^{} & \cdots \ar[r]^{}  & 0 \ar[d]^{} \ar[r]^{} & M\ar[d]^{\alpha} \ar[r]^{1} & M \ar@{=}[d]_{} \ar[r]^{} & 0 \\
   & 0  \ar[r]^{}& K_{n}\ar@{=}[d]^{}\ar[r]^{} & G_{n}\oplus P_{n-1} \ar@{=}[d]^{} \ar[r]^{} & \cdots  \ar[r]^{} &G_{1}\oplus P_{0} \ar@{=}[d]^{}\ar[r]^{\beta} & G_{0}\oplus M\ar[d]^{(1,0)} \ar[r]^{(g_{0},1)} & M \ar[d]_{} \ar[r]^{} & 0 \\
  & 0  \ar[r]^{}& K_{n}\ar[r]^{} & G_{n}\oplus P_{n-1} \ar[r]^{}& \cdots  \ar[r]^{}&G_{1}\oplus P_{0} \ar[r]^{(g_{1},f_{0})} & G_{0} \ar[r]^{} & 0 \ar[r]^{} & 0 }$$
where $\alpha=\begin{pmatrix}
0 \\
1 \\
\end{pmatrix}$ and $\beta=\begin{pmatrix}
g_{1} & f_{0} \\
0 &  -d_{0}\\
\end{pmatrix}$.
In fact, this is a short exact sequence of complexes. Thus the sequence
$0\rightarrow K_{n}\rightarrow G_{n}\oplus P_{n-1}\rightarrow \cdots\rightarrow G_{2}\oplus P_{1}\rightarrow G_{1}\oplus P_{0}\rightarrow G_{0}\rightarrow 0$ is exact with each $G_{0},~G_{i}\oplus P_{i-1}\in\mathcal{GP}$.
Because $\mathcal{GP}$ is closed under kernels of epimorphisms, $K_{n}$ is Gorenstein $(\mathcal{L}, \mathcal{A})$-projective.
\end{proof}

\begin{prop}\label{prop:2.4}{\it{$\mathcal{GP}^{\wedge}$ is closed under extensions.
}}
\end{prop}
\begin{proof} Let $0\rightarrow L \stackrel{f}\rightarrow M\stackrel{g}\rightarrow N\rightarrow 0$ be a short exact sequence with $L,~N\in\mathcal{GP}^{\wedge}$.
Without loss of generality, we assume that the $\mathcal{GP}$-resolution dimension of $L$ and that of $N$ are equivalent, denoted by $n$.
By Proposition 3.6, there exist two exact sequences
$$0\rightarrow L_{n}\stackrel{d_{n}}\rightarrow P_{n-1}\stackrel{d_{n-1}}\rightarrow \cdots\rightarrow P_{1}\stackrel{d_{1}}\rightarrow P_{0}\stackrel{d_{0}}\rightarrow L\rightarrow 0,$$
$$0\rightarrow K_{n}\stackrel{q_{n}}\rightarrow Q_{n-1}\stackrel{q_{n-1}}\rightarrow \cdots\rightarrow Q_{1}\stackrel{q_{1}}\rightarrow Q_{0}\stackrel{q_{0}}\rightarrow N\rightarrow 0$$
with $P_{i},~Q_{i}\in\mathcal{P}$ and $L_{n},K_{n}\in\mathcal{GP}$.
Next, we construct a $\mathcal{GP}$-resolution of $M$ with length at most $n$.
We have exact sequences $0\rightarrow L_{1} \rightarrow P_{0}\stackrel{d_{0}}\rightarrow L\rightarrow 0$ and $0\rightarrow K_{1} \rightarrow Q_{0}\stackrel{q_{0}}\rightarrow N\rightarrow 0$, where $L_{1}=\ker d_{0}$ and $K_{1}=\ker q_{0}$. Consider the following diagram
$$\xymatrix{
 && 0\ar[d]^{} & 0\ar[d]^{} & 0\ar[d]^{}&\\
    & 0  \ar[r]^{}& L_{1}\ar[d]^{}\ar[r]^{} & M_{1} \ar[d]^{} \ar[r]^{} & K_{1} \ar[d]^{} \ar[r]^{} & 0 \\
   & 0  \ar[r]^{}& P_{0}\ar[d]^{d_{0}}\ar[r]^{\alpha} & P_{0}\oplus Q_{0}  \ar[d]^{(x,y)} \ar[r]^{(0,1)} &  Q_{0} \ar[d]^{q_{0}} \ar[r]^{} & 0 \\
  & 0  \ar[r]^{}& L\ar[d]^{}\ar[r]^{f} & M\ar[d]^{}\ar[r]^{g}& N \ar[d]^{}\ar[r]^{} & 0\\
   && 0 & 0 & 0 &}$$
where $\alpha=\begin{pmatrix}
1 \\
0 \\
\end{pmatrix}$ and $M_{1}=\ker (x,y)$. Put $x=fd_{0}$. Since $Q_{0}$ is projective, there exists a morphism $h:Q_{0}\rightarrow M$ such that $q_{0}=gh$. Let $y=h$. One easily checks that two below squares are commutative. By Snake Lemma, the first row is exact and $(x,y)$ is surjective. Repeating this process, we obtain an exact sequence
$$0\rightarrow M_{n}\rightarrow P_{n-1}\oplus Q_{n-1}\rightarrow \cdots\rightarrow P_{1}\oplus Q_{1}\rightarrow P_{0}\oplus Q_{0}\rightarrow M\rightarrow 0$$ with each $P_{i}\oplus Q_{i}\in\mathcal{P}$. Indeed, $M_{n}\in\mathcal{GP}$
since it occurs in the exact sequence $0\rightarrow L_{n} \rightarrow M_{n}\rightarrow K_{n}\rightarrow 0$ with $L_{n},~K_{n}\in\mathcal{GP}$, and $\mathcal{GP}$ is closed under extensions.
\end{proof}

We recall Hovey-Gillespie correspondence [10, Theorem 3.3].
Let $(\mathcal{B},\mathcal{E})$ be an exact category with an exact model structure.
Let $\mathcal{Q}$ be the class of cofibrant objects, $\mathcal{R}$ the class
of fibrant objects and $\mathcal{W}$ the class of trivial objects. Then $\mathcal{W}$ is a thick subcategory of $\mathcal{B}$ and both $(\mathcal{Q},\mathcal{R\cap W})$ and $(\mathcal{Q\cap W},\mathcal{R})$ are complete cotorsion pairs in $\mathcal{B}$. If $(\mathcal{B},\mathcal{E})$ is weakly idempotent complete, then the converse holds. That is, given a thick subcategory $\mathcal{W}$ and classes $\mathcal{Q}$ and $\mathcal{R}$ making $(\mathcal{Q},\mathcal{R\cap W})$ and $(\mathcal{Q\cap W},\mathcal{R})$ complete cotorsion pairs, then there is an
exact model structure on $\mathcal{B}$ where $\mathcal{Q}$ is the class of cofibrant objects, $\mathcal{R}$ is the class of fibrant objects and $\mathcal{W}$ is the class of trivial objects. Hence, we denote an exact model structure as a triple $(\mathcal{Q},\mathcal{W},\mathcal{R})$, and call it a Hovey triple.
If readers want to further learn model structures, please refer to [10, 11, 12, 13, 16].

We have the following example.

\begin{exa}\label{prop:2.4}{\rm{$(\mathcal{GP},\mathcal{P}^{\wedge})$ is a left AB-context in $R$-Mod. Furthermore, $(\mathcal{GP},\mathcal{P})$ is a strong left Frobenius pair by [1, Proposition 6.10]. In addition, $\mathcal{GP}$ is a Frobenius category with $\mathcal{P}$ a class of relative projective-injective objects and $\underline{\mathcal{GP}}:= \mathcal{GP}/\mathcal{P}$ is a triangulated category. There exists a Frobenius model structure $(\mathcal{GP},\mathcal{P},\mathcal{GP})$ on the exact category $\mathcal{GP}$, where $\mathcal{GP}$ is the class of cofibrant objects, $\mathcal{P}$ is the class of trivial objects and $\mathcal{GP}$ is the class of fibrant objects.
There exists an exact model structure $(\mathcal{GP},\mathcal{P}^{\wedge},\mathcal{GP}^{\wedge})$ on the exact category $\mathcal{GP}^{\wedge}$, where $\mathcal{GP}$ is the class of cofibrant objects, $\mathcal{P}^{\wedge}$ is the class of trivial objects and $\mathcal{GP}^{\wedge}$ is the class of fibrant objects.
}}
\end{exa}

\begin{rem}\label{prop:2.4}{\rm{The canonical example of a complete duality pair is the level duality pair $(\mathcal{L},\mathcal{A})$ over any ring, given in [3], where $\mathcal{L}$ represents the class of level modules and $\mathcal{A}$ represents the class of absolutely clean modules. In this case, the class of Gorenstein $(\mathcal{L}, \mathcal{A})$-projective modules coincides with the class of Gorenstein AC-projective modules, denoted by $\mathcal{G}_{ac}\mathcal{P}$ (refer to [3]). We obtain that $(\mathcal{G}_{ac}\mathcal{P},\mathcal{P}^{\wedge})$ is a left AB-context in $R$-Mod. $(\mathcal{G}_{ac}\mathcal{P},\mathcal{P})$ is a strong left Frobenius pair
by [1, Corollary 6.11]. Furthermore, $\mathcal{G}_{ac}\mathcal{P}$ is a Frobenius category with $\mathcal{P}$ a class of relative projective-injective objects and hence $\underline{\mathcal{G}_{ac}\mathcal{P}}:= \mathcal{G}_{ac}\mathcal{P}/\mathcal{P}$ is a triangulated category. Moreover, there exists a Frobenius model structure $(\mathcal{G}_{ac}\mathcal{P},\mathcal{P},\mathcal{G}_{ac}\mathcal{P})$ on the exact category $\mathcal{G}_{ac}\mathcal{P}$ and there exists an exact model structure $(\mathcal{G}_{ac}\mathcal{P},\mathcal{P}^{\wedge},\mathcal{G}_{ac}\mathcal{P}^{\wedge})$ on the exact category $\mathcal{G}_{ac}\mathcal{P}^{\wedge}$.
}}
\end{rem}

Actually, we can get the results that $(\mathcal{GP},\mathcal{F})$ and $(\mathcal{GP},\mathcal{L})$ are left Frobenius pairs when $\mathcal{F}\subseteq\mathcal{GP}$ and $\mathcal{L}\subseteq\mathcal{GP}$ respectively. In the following, we give some characterizations of rings via Frobenius pairs.

Recall $R$ is a left perfect ring if every flat left $R$-module is projective.

\begin{prop}\label{prop:2.4}{\it{The following conditions are equivalent for any ring $R$.

(1) $(\mathcal{GP},\mathcal{F})$ is a strong left Frobenius pair.

(2) $(\mathcal{GP},\mathcal{F})$ is a left Frobenius pair.

(3) $R$ is a left perfect ring.
}}
\end{prop}
\begin{proof} (1) $\Rightarrow$ (2) is straightforward.

(2) $\Rightarrow$ (3) Since $(\mathcal{GP},\mathcal{F})$ is a left Frobenius pair,
$\mathcal{F}$ is a relative cogenerator in $\mathcal{GP}$ and
hence $\mathcal{F}\subseteq\mathcal{GP}$. For any $M\in\mathcal{F}$, there exists a short exact sequence $0\rightarrow M\rightarrow P\rightarrow G\rightarrow 0$ with $P\in\mathcal{P}$ and $G\in\mathcal{GP}$ by Proposition 3.1.
Since $\mathrm{Ext}^{1}_{R}(G,M)=0$, $0\rightarrow M\rightarrow P\rightarrow G\rightarrow 0$ splits. Thus $M$ is projective. So $R$ is a left perfect ring.

When $R$ is a left perfect ring, $\mathcal{F}$ is exactly $\mathcal{P}$.
By [1, Proposition 6.10], (3) implies (1).
\end{proof}

Immediately, we obtain the next simple result.

\begin{prop}\label{prop:2.4}{\it{The following conditions are equivalent for any ring $R$.

(1) $(\mathcal{GP},\mathcal{L})$ is a strong left Frobenius pair.

(2) $(\mathcal{GP},\mathcal{L})$ is a left Frobenius pair.

(3) $\mathcal{L}=\mathcal{P}$.

In this case, $R$ is a left perfect ring.
}}
\end{prop}

It is known that $R$ is a right coherent ring if and only if all level left $R$-modules are flat. In the following, we give an example.

\begin{exa}\label{prop:2.4}{\rm{Let $(\mathcal{L},\mathcal{A})$ be the level duality pair  over any ring $R$. Then the following conditions are equivalent.

(1) $(\mathcal{G}_{ac}\mathcal{P},\mathcal{L})$ is a strong left Frobenius pair.

(2) $(\mathcal{G}_{ac}\mathcal{P},\mathcal{L})$ is a left Frobenius pair.

(3) $\mathcal{L}=\mathcal{P}$.

(4) $R$ is a right coherent and left perfect ring.
}}
\end{exa}

Recall an $R$-module $M$ is called Gorenstein $(\mathcal{L}, \mathcal{A})$-flat if there exists a $\mathcal{A}{\otimes}_{R}-$-exact exact sequence $$\mathbb{F}:\cdots\rightarrow F_{1}\rightarrow F_{0}\rightarrow F_{-1}\rightarrow F_{-2}\rightarrow \cdots$$
with each $F_{i}\in\mathcal{F}$ such that $M\cong$~Ker$(F_{-1}\rightarrow F_{-2})$. $\mathcal{GF}$ denotes the class of all Gorenstein $(\mathcal{L}, \mathcal{A})$-flat modules (see [13]).

Many scholars have been interested in when the Gorenstein projective module is Gorenstein flat in Gorenstein homological algebra for two decades. So far, this problem is still open. Moreover, relative results were given under certain circumstance. But
when the Gorenstein $(\mathcal{L}, \mathcal{A})$-projective module is
Gorenstein $(\mathcal{L}, \mathcal{A})$-flat is trivial.

\begin{rem}\label{prop:2.4}{\rm{By [3, Theorem A.6], $\mathrm{Hom}_{R}(\mathbb{P},\mathcal{L})$ is exact if and only if $\mathcal{A}{\otimes}_{R}\mathbb{P}$ is exact, where $\mathbb{P}$ is an exact sequence of projective modules. Then it follows from definitions of Gorenstein $(\mathcal{L}, \mathcal{A})$-projective and Gorenstein $(\mathcal{L}, \mathcal{A})$-flat modules that
a Gorenstein $(\mathcal{L}, \mathcal{A})$-projective module is
Gorenstein $(\mathcal{L}, \mathcal{A})$-flat over any ring. In particular, a Gorenstein AC-projective module is
Gorenstein AC-flat.
}}
\end{rem}

In [18], authors introduced and studied a kind of Gorenstein $(\mathcal{X}, \mathcal{Y})$-flat module with respect to two classes of modules $\mathcal{X}$ and $\mathcal{Y}$.
An $R$-module $M$ is called Gorenstein $(\mathcal{X}, \mathcal{Y})$-flat if there exists a $\mathcal{Y}{\otimes}_{R}-$-exact exact sequence $$\mathbb{X}:\cdots\rightarrow X_{1}\rightarrow X_{0}\rightarrow X_{-1}\rightarrow X_{-2}\rightarrow \cdots$$
with each $X_{i}\in\mathcal{X}$ such that $M\cong$~Ker$(X_{-1}\rightarrow X_{-2})$. Authors use $\mathcal{GF_{(X,Y)}}(R)$ to denote the class of all Gorenstein $(\mathcal{X}, \mathcal{Y})$-flat modules.

\begin{rem}\label{prop:2.4}{\rm{It is well known that $\mathcal{P}\subseteq\mathcal{GP}\subseteq\mathcal{GF}$ and $\mathcal{P}\subseteq\mathcal{F}\subseteq\mathcal{GF}$.
If $\mathcal{X}=\mathcal{L}$ and $\mathcal{Y}=\mathcal{A}$, then $\mathcal{P}\subseteq\mathcal{F}\subseteq\mathcal{L}\subseteq\mathcal{GF_{(L,A)}}(R)$ and
$\mathcal{GP}\subseteq\mathcal{GF}\subseteq\mathcal{GF_{(L,A)}}(R)$.
Furthermore, if $\mathcal{L}$ is the class of flat left $R$-modules and $\mathcal{A}$ is the class of injective right $R$-modules over a right noetherian ring, then $\mathcal{GF}=\mathcal{GF_{(L,A)}}(R)$.
}}
\end{rem}

Complete cotorsion pairs play an important role in constructing model structures.
When $(\mathcal{L},\mathcal{A})$ is a perfect duality pair, $(\mathcal{L},\mathcal{L}^{\perp})$ is a perfect cotorsion pair. This fact builds the bridge between duality pairs and cotorsion pairs. In [18, Proposition 2.18], authors gave some equivalent characterizations of $\mathcal{GF_{(X,Y)}}(R)$ under some conditions, where $(\mathcal{X},\mathcal{Y})$ is a complete duality pair. We introduce the definition of Gorenstein $\mathcal{X}$-objects with respect to a cotorsion pair and use $\mathcal{G}(\mathcal{X})$ to represent the class of Gorenstein $\mathcal{X}$-objects
in [19]. We observe that $\mathcal{GF_{(X,Y)}}(R)$ is exactly $\mathcal{G}(\mathcal{X})$ under the conditions of [18, Proposition 2.18]. So we will obtain exact model structures associated to $\mathcal{GF_{(X,Y)}}(R)$.

Let $(\mathcal{C}, \mathcal{D})$ be a complete and hereditary cotorsion pair. For convenience, recall that an $R$-module $M$ is said to be a Gorenstein $\mathcal{C}$-object if there exists a $\mathrm{Hom}_{R}(-,\mathcal{C\cap D})$-exact exact sequence $$\mathbb{C}:\cdots\rightarrow C_{1}\rightarrow C_{0}\rightarrow C_{-1}\rightarrow C_{-2}\rightarrow \cdots$$
with each $C_{i}\in\mathcal{C}$ such that $M\cong$~Ker$(C_{-1}\rightarrow C_{-2})$ (see [19]).

In the following, we give exact model structures associated to $\mathcal{GF_{(X,Y)}}(R)$.

\begin{thm}\label{prop:2.4}{\it{Let $(\mathcal{X},\mathcal{Y})$ be a complete duality pair, $\mathcal{X}$ closed under kernels of epimorphisms and $\mathrm{Tor}^{R}_{i}(Y,X)=0$ for all $Y\in\mathcal{Y}$, $X\in\mathcal{X}$ and $i\geq 1$. Then $(\mathcal{X},\mathcal{X\cap X^{\perp}})$ is a left Frobenius pair and
$(\mathcal{GF_{(X,Y)}}(R),(\mathcal{X\cap X^{\perp}})^{\wedge})$ is a left AB-context in $R$-Mod. Furthermore, assume that $\mathcal{X^{\perp}}$ is closed under kernels of epimorphisms, then $(\mathcal{GF_{(X,Y)}}(R)\cap \mathcal{X}^{\perp},\mathcal{X\cap X^{\perp}})$ is a strong left Frobenius pair. In addition, $\mathcal{GF_{(X,Y)}}(R)\cap \mathcal{X}^{\perp}$ is a Frobenius category with $\mathcal{X\cap X^{\perp}}$ a class of relative projective-injective objects and hence $\underline{\mathcal{GF_{(X,Y)}}(R)\cap \mathcal{X}^{\perp}}:= \mathcal{GF_{(X,Y)}}(R)\cap \mathcal{X}^{\perp}/\mathcal{X\cap X^{\perp}}$ is a triangulated category. There exists a Frobenius model structure $(\mathcal{GF_{(X,Y)}}(R)\cap \mathcal{X}^{\perp},\mathcal{X\cap X^{\perp}},\mathcal{GF_{(X,Y)}}(R)\cap \mathcal{X}^{\perp})$ on the exact category $\mathcal{GF_{(X,Y)}}(R)\cap \mathcal{X}^{\perp}$.
Moreover, there exists an exact model structure $(\mathcal{GF_{(X,Y)}}(R)\cap \mathcal{X}^{\perp},(\mathcal{X\cap X^{\perp}})^{\wedge},(\mathcal{GF_{(X,Y)}}(R)\cap \mathcal{X}^{\perp})^{\wedge})$ on the exact category $(\mathcal{GF_{(X,Y)}}(R)\cap \mathcal{X}^{\perp})^{\wedge}$, where $\mathcal{GF_{(X,Y)}}(R)\cap \mathcal{X}^{\perp}$ is the class of cofibrant objects, $(\mathcal{X\cap X^{\perp}})^{\wedge}$ is the class of trivial objects and $(\mathcal{GF_{(X,Y)}}(R)\cap \mathcal{X}^{\perp})^{\wedge}$ is the class of fibrant objects.
}}
\end{thm}
\begin{proof}For one thing, since $(\mathcal{X},\mathcal{Y})$ is a perfect duality pair,
$(\mathcal{X},\mathcal{X}^{\perp})$ is a perfect cotorsion pair by [15, Theorem 3.1]. Naturally, $(\mathcal{X},\mathcal{X}^{\perp})$ is a complete cotorsion pair. For another, since $\mathcal{X}$ is closed under kernels of epimorphisms, $(\mathcal{X},\mathcal{X}^{\perp})$ is a hereditary cotorsion pair by [12, Lemma 2.4].
Namely, $(\mathcal{X},\mathcal{X}^{\perp})$ is a complete and hereditary cotorsion pair.
Since $(\mathcal{X},\mathcal{X}^{\perp})$ is hereditary, $\mathrm{Ext}^{i}_{R}(X,Z)=0$ for all $X\in\mathcal{X}$, $Z\in\mathcal{X}^{\perp}$ and $i\geq 1$. Then
id$_{\mathcal{X}}\mathcal{X\cap X^{\perp}}=0$, namely, $\mathcal{X\cap X^{\perp}}$ is $\mathcal{X}$-injective. Obviously, $\mathcal{X\cap X^{\perp}}\subseteq \mathcal{X}$.
For any $X\in\mathcal{X}$, there exists a short exact sequence $0\rightarrow X\rightarrow Z\rightarrow X^{'}\rightarrow 0$ with $Z\in\mathcal{X^{\perp}}$ and $X^{'}\in\mathcal{X}$ by the completeness of $(\mathcal{X},\mathcal{X}^{\perp})$.
Because $\mathcal{X}$ is closed under extensions, $Z\in\mathcal{X\cap X^{\perp}}$.
Thus $\mathcal{X\cap X^{\perp}}$ is a relative cogenerator in $\mathcal{X}$. It is clear that $\mathcal{X}$ is left thick and $\mathcal{X\cap X^{\perp}}$ is closed under direct summands. So $(\mathcal{X},\mathcal{X\cap X^{\perp}})$ is a left Frobenius pair.
By [18, Proposition 2.18], $M$ is in $\mathcal{GF_{(X,Y)}}(R)$ if and only if there exists an exact sequence of left $R$-modules in $\mathcal{X}$
$$\mathbb{X}:\cdots\rightarrow X_{1}\rightarrow X_{0}\rightarrow X_{-1}\rightarrow X_{-2}\rightarrow \cdots$$
such that $M\cong$~Ker$(X_{-1}\rightarrow X_{-2})$ and   $\mathrm{Hom}_{R}(-,\mathcal{X\cap X^{\perp}})$ leaves the sequence exact.
Comparing with the definition of Gorenstein $\mathcal{X}$-objects, we obtain
that $\mathcal{GF_{(X,Y)}}(R)$ is exactly $\mathcal{G}(\mathcal{X})$. By [5, Proposition 4.1], $(\mathcal{GF_{(X,Y)}}(R),(\mathcal{X\cap X^{\perp}})^{\wedge})$ is a left AB-context in $R$-Mod. Assume that $\mathcal{X^{\perp}}$ is closed under kernels of epimorphisms. By [5, Theorem 4.2], $(\mathcal{GF_{(X,Y)}}(R)\cap \mathcal{X}^{\perp},\mathcal{X\cap X^{\perp}})$ is a strong left Frobenius pair, $\mathcal{GF_{(X,Y)}}(R)\cap \mathcal{X}^{\perp}$ is a Frobenius category with $\mathcal{X\cap X^{\perp}}$ a class of relative projective-injective objects and $\underline{\mathcal{GF_{(X,Y)}}(R)\cap \mathcal{X}^{\perp}}:= \mathcal{GF_{(X,Y)}}(R)\cap \mathcal{X}^{\perp}/\mathcal{X\cap X^{\perp}}$ is a triangulated category. Thus there exists a Frobenius model structure $(\mathcal{GF_{(X,Y)}}(R)\cap\mathcal{X}^{\perp},\mathcal{X\cap X^{\perp}},\mathcal{GF_{(X,Y)}}(R)\cap \mathcal{X}^{\perp})$ on the exact category $\mathcal{GF_{(X,Y)}}(R)\cap \mathcal{X}^{\perp}$.
At the same time, there exists an exact model structure $(\mathcal{GF_{(X,Y)}}(R)\cap \mathcal{X}^{\perp},(\mathcal{X\cap X^{\perp}})^{\wedge},(\mathcal{GF_{(X,Y)}}(R)\cap \mathcal{X}^{\perp})^{\wedge})$ on the exact category $(\mathcal{GF_{(X,Y)}}(R)\cap \mathcal{X}^{\perp})^{\wedge}$, where $\mathcal{GF_{(X,Y)}}(R)\cap \mathcal{X}^{\perp}$ is the class of cofibrant objects, $(\mathcal{X\cap X^{\perp}})^{\wedge}$ is the class of trivial objects and $(\mathcal{GF_{(X,Y)}}(R)\cap \mathcal{X}^{\perp})^{\wedge}$ is the class of fibrant objects.
\end{proof}

\section{strongly Gorenstein $(\mathcal{L}, \mathcal{A})$-projective modules}
In this section, we introduce and investigate strongly Gorenstein $(\mathcal{L}, \mathcal{A})$-projective modules with respect to a fixed complete duality pair $(\mathcal{L}, \mathcal{A})$.
Meanwhile, we attempt to construct some relevant model structures and recollements.

\begin{df}\label{df:2.1} {\rm An $R$-module $M$ is called strongly Gorenstein $(\mathcal{L}, \mathcal{A})$-projective if there exists a $\mathrm{Hom}_{R}(-,\mathcal{L})$-exact exact sequence $$\mathbb{P}:\cdots\rightarrow P\stackrel{f}\rightarrow P\stackrel{f}\rightarrow P\stackrel{f}\rightarrow P\stackrel{f}\rightarrow \cdots$$
with $P\in\mathcal{P}$ such that $M\cong\ker f$.}
\end{df}

We use $\mathcal{SGP}$ to denote the class of all strongly Gorenstein $(\mathcal{L}, \mathcal{A})$-projective modules. Immediately,
we get $\mathcal{SGP}\subseteq\mathcal{GP}$. It is easy to obtain
that $\mathcal{SGP}$ is closed under direct sums.

In the following, we give some homological characterizations of $\mathcal{SGP}$.

\begin{prop}\label{prop:2.4}{\it{The following conditions are equivalent for an $R$-module $M$.

(1) $M$ is strongly Gorenstein $(\mathcal{L}, \mathcal{A})$-projective.

(2) There exists a $\mathrm{Hom}_{R}(-,\mathcal{L})$-exact exact sequence $0\rightarrow M\rightarrow P\rightarrow M\rightarrow 0$ with $P\in\mathcal{P}$.

(3) There exists an exact sequence $0\rightarrow M\rightarrow P\rightarrow M\rightarrow 0$ with $P\in\mathcal{P}$ and $\mathrm{Ext}^{1}_{R}(M,L)=0$ for all $L\in\mathcal{L}$.

(4) There exists an exact sequence $0\rightarrow M\rightarrow P\rightarrow M\rightarrow 0$ with $P\in\mathcal{P}$ and $\mathrm{Ext}^{i}_{R}(M,L)=0$ for all $L\in\mathcal{L}$ and $i\geq 1$.
}}
\end{prop}
\begin{proof} It is straightforward.
\end{proof}

\begin{prop}\label{prop:2.4}{\it{Every projective module $P$ is strongly Gorenstein $(\mathcal{L}, \mathcal{A})$-projective.
}}
\end{prop}
\begin{proof} For a projective module $P$, there exists an exact sequence $0\rightarrow P\stackrel{\alpha}\rightarrow P\oplus P\stackrel{(0,1)}\rightarrow P\rightarrow 0$ with $\alpha=\begin{pmatrix}
1\\
0 \\
\end{pmatrix}$. It is $\mathrm{Hom}_{R}(-,\mathcal{L})$-exact since it splits. By Proposition 4.2, $P$ is strongly Gorenstein $(\mathcal{L}, \mathcal{A})$-projective.
\end{proof}

From now on, we have $\mathcal{P}\subseteq\mathcal{SGP}\subseteq\mathcal{GP}
\subseteq\mathcal{GF}\subseteq\mathcal{GF_{(L,A)}}(R)$.

\begin{prop}\label{prop:2.4}{\it{Let $G$ be a strongly Gorenstein $(\mathcal{L}, \mathcal{A})$-projective module. Then the following statements hold.

(1) $G^{\perp}$ is a thick subcategory of $R$-Mod.

(2) $(^{\perp}(G^{\perp}),G^{\perp})$ is a projective cotorsion pair, cogenerated by a set.
}}
\end{prop}
\begin{proof} (1) Obviously, $G^{\perp}$ is closed under direct summands and extensions. Let $0\rightarrow L\rightarrow M\rightarrow N\rightarrow 0$ be an exact sequence.
At first, we assume that $L\in G^{\perp}$ and $M\in G^{\perp}$. Then there is an exact sequence $0=\mathrm{Ext}^{1}_{R}(G,M)\rightarrow \mathrm{Ext}^{1}_{R}(G,N)\rightarrow \mathrm{Ext}^{2}_{R}(G,L)$. Since $G$ is strongly Gorenstein $(\mathcal{L}, \mathcal{A})$-projective, there exists an exact sequence $0\rightarrow G\rightarrow P\rightarrow G\rightarrow 0$ with $P\in\mathcal{P}$. Then we get an exact sequence $0=\mathrm{Ext}^{1}_{R}(G,L)\rightarrow \mathrm{Ext}^{2}_{R}(G,L)\rightarrow \mathrm{Ext}^{2}_{R}(P,L)=0$. Then $\mathrm{Ext}^{2}_{R}(G,L)=0$ and hence $\mathrm{Ext}^{1}_{R}(G,N)=0$, that is $N\in G^{\perp}$. Next, we assume that $M\in G^{\perp}$ and $N\in G^{\perp}$. Then there is an exact sequence $0=\mathrm{Ext}^{1}_{R}(G,N)\rightarrow \mathrm{Ext}^{2}_{R}(G,L)\rightarrow \mathrm{Ext}^{2}_{R}(G,M)$=0. Hence, $\mathrm{Ext}^{2}_{R}(G,L)=0$. There is an exact sequence $0=\mathrm{Ext}^{1}_{R}(P,L)\rightarrow \mathrm{Ext}^{1}_{R}(G,L)\rightarrow \mathrm{Ext}^{2}_{R}(G,L)$=0. Then $\mathrm{Ext}^{1}_{R}(G,L)=0$. So $G^{\perp}$ is a thick subcategory of $R$-Mod.

(2) Put a set $\mathcal{S}=\{G\}$. By [14, Theorem 6.11], $(^{\perp}(G^{\perp}),G^{\perp})$ is a complete cotorsion pair, cogenerated by the set $\mathcal{S}$. Since $\mathcal{P}\subseteq\mathcal{L}$, $\mathcal{P}\subseteq G^{\perp}$. It follows from [12, Proposition 3.7] that $(^{\perp}(G^{\perp}),G^{\perp})$ is a projective cotorsion pair.
\end{proof}

Of course, $(^{\perp}(G^{\perp}),G^{\perp})$ is a hereditary cotorsion pair.

\begin{thm}\label{prop:2.4}{\it{Let $G$ be a strongly Gorenstein $(\mathcal{L}, \mathcal{A})$-projective module.
Then there exists a hereditary abelian model structure $(\mathcal{GP},\mathcal{W},G^{\perp})$ on $R$-Mod, where $\mathcal{W}$ can be described in the two following ways:$$\mathcal{W}=\{W\in R-\mathrm{Mod}\mid \mathrm{\exists~an~exact~sequence}~0\rightarrow W\rightarrow A\rightarrow B\rightarrow 0~\mathrm{with}~B\in {^{\perp}(G^{\perp})},~A\in\mathcal{GP}^\perp\}$$
$$~~=\{W\in R-\mathrm{Mod}\mid \mathrm{\exists~an~exact~sequence}~0\rightarrow A^{'}\rightarrow B^{'}\rightarrow W\rightarrow 0~\mathrm{with}~B^{'}\in {^{\perp}(G^{\perp})},~A^{'}\in\mathcal{GP}^\perp\}.$$
}}
\end{thm}
\begin{proof} By [13, Theorem 4.9], we know that $(\mathcal{GP}, \mathcal{GP}^\perp)$ is a projective cotorsion pair.
Because $(\mathcal{GP}, \mathcal{GP}^\perp)$  and $(^{\perp}(G^{\perp}),G^{\perp})$ are two projective cotorsion pairs, $\mathcal{GP}\cap\mathcal{GP}^\perp=\mathcal{P}={^{\perp}(G^{\perp})}\cap G^{\perp}$.
Since $G\in\mathcal{SGP}\subseteq\mathcal{GP}$, $\mathcal{GP}^\perp\subseteq G^{\perp}$ and $^{\perp}(G^{\perp})\subseteq\mathcal{GP}$. It immediately follows from [11, Theorem 1.1].
\end{proof}

\begin{rem}\label{prop:2.4}{\rm{When one chooses different strongly Gorenstein $(\mathcal{L}, \mathcal{A})$-projective module $G$, the different model structure may be obtained.
}}
\end{rem}

We have the following result, which gives the relations between $\mathcal{SGP}$ and $\mathcal{GP}$.

\begin{prop}\label{prop:2.4}{\it{The following conditions are equivalent.

(1) $M$ is a Gorenstein $(\mathcal{L}, \mathcal{A})$-projective module.

(2) $M$ is a direct summand of some strongly Gorenstein $(\mathcal{L}, \mathcal{A})$-projective module.

(3) There exists a strongly Gorenstein $(\mathcal{L}, \mathcal{A})$-projective module $G$ such that $M\in{^{\perp}(G^{\perp})}$.
}}
\end{prop}
\begin{proof}
$(1) \Rightarrow (2)$ Since $M$ is Gorenstein $(\mathcal{L}, \mathcal{A})$-projective, there exists a $\mathrm{Hom}_{R}(-,\mathcal{L})$-exact exact sequence $$\mathbb{P}:\cdots\rightarrow P_{1}\stackrel{d_{1}}\rightarrow P_{0}\stackrel{d_{0}}\rightarrow P_{-1}\stackrel{d_{-1}}\rightarrow P_{-2}\rightarrow \cdots$$
with each $P_{i}\in\mathcal{P}$ such that $M\cong\ker d_{-1}$.
Next, we sum up all shifts of $\mathbb{P}$ and hence we get a $\mathrm{Hom}_{R}(-,\mathcal{L})$-exact exact sequence $$\cdots\rightarrow \oplus P_{i}\stackrel{\oplus d_{i}}\rightarrow \oplus P_{i}\stackrel{\oplus d_{i}}\rightarrow \oplus P_{i}\stackrel{\oplus d_{i}}\rightarrow \oplus P_{i}\rightarrow \cdots$$
Then $\ker(\oplus d_{i})$ is a strongly Gorenstein $(\mathcal{L}, \mathcal{A})$-projective module and it is obvious that $\oplus \ker d_{i}\cong\ker(\oplus d_{i})$. So $M$ is a direct summand of the strongly Gorenstein $(\mathcal{L}, \mathcal{A})$-projective module $\ker(\oplus d_{i})$.

$(2) \Rightarrow (3)$ Let $M$ be a direct summand of a strongly Gorenstein $(\mathcal{L}, \mathcal{A})$-projective module $G$. Since $G\in{^{\perp}(G^{\perp})}$, $M\in{^{\perp}(G^{\perp})}$.

$(3) \Rightarrow (1)$ By Proposition 4.4, we have an exact sequence $0\rightarrow M\rightarrow P_0\rightarrow L_0\rightarrow 0$ with $P_0\in G^{\perp}$ and $L_0\in {^{\perp}(G^{\perp})}$. Thus $P_0\in G^{\perp}\cap{^{\perp}(G^{\perp})}=\mathcal{P}$. Continuing this process, we get a long exact sequence $0\rightarrow M\rightarrow P_0\stackrel{f_0}\rightarrow P_{-1}\stackrel{f_{-1}}\rightarrow P_{-2}\stackrel{f_{-2}}\rightarrow \cdots$ with $P_{i}\in \mathcal{P}$ and $\ker f_{i}\in {^{\perp}(G^{\perp})}$. Since $\mathcal{L}\subseteq G^{\perp}$, this long exact sequence is $\mathrm{Hom}_{R}(-,\mathcal{L})$-exact.
Note that $(^{\perp}(G^{\perp}),G^{\perp})$ is a hereditary cotorsion pair, so $\mathrm{Ext}^{i}_{R}(M,L)=0$ for any $L\in\mathcal{L}$ and $i\geq 1$. By Proposition 3.1, $M$ is a Gorenstein $(\mathcal{L}, \mathcal{A})$-projective module.
\end{proof}

In the following, we recall some notions and basic facts.

Let $\mathcal{S}$ be a class of $R$-modules. A complex $X$ is in $dw\mathcal{S}$ if $X_{j}\in\mathcal{S}$ for any $j\in \mathbb{Z}$.

An exact complex $X$ is in $ex\mathcal{S}$ if $X_{j}\in\mathcal{S}$ for any $j\in \mathbb{Z}$.

An exact complex $X$ is in $\widetilde{\mathcal{S}}$ if $Z_{j}(X)\in\mathcal{S}$ for any $j\in \mathbb{Z}$.

Let $(\mathcal{S},\mathcal{T})$ be a cotorsion pair in $R$-Mod. A complex $Y$ is a dg$\mathcal{S}$ complex if each $Y_{n}\in\mathcal{S}$ and if each map $Y\rightarrow U$ is null homotopic for each complex $U\in\mathcal{\widetilde{T}}$. The definition of a dg$\mathcal{T}$ complex is dual. We use $dg\mathcal{S}$ and $dg\mathcal{T}$ to denote the class of all dg$\mathcal{S}$ complexes and the class of all dg$\mathcal{T}$ complexes respectively.

By Proposition 4.4, we know that for a strongly Gorenstein $(\mathcal{L}, \mathcal{A})$-projective module $G$, $(^{\perp}(G^{\perp}),G^{\perp})$ is a projective cotorsion pair, cogenerated by a set $\{G\}$. It immediately follows from [12, Proposition 7.3] that there are six projective cotorsion pairs in Ch$(R)$, cogenerated by sets, namely, $(dw^{\perp}(G^{\perp}), (dw^{\perp}(G^{\perp}))^\perp)$, $(ex^{\perp}(G^{\perp}), (ex^{\perp}(G^{\perp}))^\perp)$,
$(\widetilde{^{\perp}(G^{\perp})}, dg G^{\perp})$, $(^\perp(dw G^{\perp}), dw G^{\perp})$, $(^\perp(ex G^{\perp}), ex G^{\perp})$,
and $(dg^{\perp}(G^{\perp}), \widetilde{G^{\perp}})$. Of course, for $(\mathcal{P}, R$-Mod$)$, there exist four distinct projective cotorsion pairs in Ch$(R)$. They are
$(dw\mathcal{P}, (dw\mathcal{P})^\perp)$, $(ex\mathcal{P}, (ex\mathcal{P})^\perp)$,
$(\widetilde{\mathcal{P}},$ Ch$(R))$, and $(dg\mathcal{P}, \mathcal{E})$. In the following, we will use these projective cotorsion pairs to construct some recollements. For this goal, we recall some notions.

Let $\mathcal{G}$ be a class of $R$-modules and $(\mathcal{S},\mathcal{T})$ a cotorsion pair in $R$-Mod. we use $D(R)$ to denote the derived category of $R$-modules,
$K(\mathcal{G})$ to denote the homotopy category of $\mathcal{G}$,
$K_{ex}(\mathcal{G})$ to denote the homotopy category consisting of exact complexes of $\mathcal{G}$, $K(dg\mathcal{S})$ to denote the homotopy category consisting of dg$\mathcal{S}$ complexes, $K_{ex}(dg\mathcal{S})$ to denote the homotopy category consisting of exact dg$\mathcal{S}$ complexes,
$K(^\perp(dw\mathcal{T}))$ to denote the homotopy category consisting of complexes in $^\perp(dw\mathcal{T})$, and $K(^\perp(ex\mathcal{T}))$ to denote the homotopy category consisting of complexes in $^\perp(ex\mathcal{T})$.

\begin{thm}\label{prop:2.4}{\it{There exist five recollements associated to a strongly Gorenstein $(\mathcal{L}, \mathcal{A})$-projective module $G$, where three recollements are relative to the derived category $D(R)$, as follows:

(1)$$\xymatrix{K_{ex}(^{\perp}(G^{\perp}))\ar^-{}[r]&K(^{\perp}(G^{\perp}))\ar^-{}[r]
\ar^-{}@/^1.2pc/[l]\ar_-{}@/_1.6pc/[l]
&D(R)\ar^-{}@/^1.2pc/[l]\ar_-{}@/_1.6pc/[l]}$$

(2)$$\xymatrix{K_{ex}(dg^{\perp}(G^{\perp}))\ar^-{}[r]&K(dg^{\perp}(G^{\perp}))\ar^-{}[r]
\ar^-{}@/^1.2pc/[l]\ar_-{}@/_1.6pc/[l]
&D(R)\ar^-{}@/^1.2pc/[l]\ar_-{}@/_1.6pc/[l]}$$

(3)$$\xymatrix{K(^\perp(dw G^{\perp}))\ar^-{}[r]&K(^\perp(ex G^{\perp}))\ar^-{}[r]
\ar^-{}@/^1.2pc/[l]\ar_-{}@/_1.6pc/[l]
&D(R)\ar^-{}@/^1.2pc/[l]\ar_-{}@/_1.6pc/[l]}$$

(4)$$\xymatrix{K(\mathcal{P})\ar^-{}[r]&K(^{\perp}(G^{\perp}))\ar^-{}[r]
\ar^-{}@/^1.2pc/[l]\ar_-{}@/_1.6pc/[l]
&K(^\perp(dw G^{\perp}))\ar^-{}@/^1.2pc/[l]\ar_-{}@/_1.6pc/[l]}$$

(5)$$\xymatrix{K_{ex}(\mathcal{P})\ar^-{}[r]&K(^{\perp}(G^{\perp}))\ar^-{}[r]
\ar^-{}@/^1.2pc/[l]\ar_-{}@/_1.6pc/[l]
&K(^\perp(ex G^{\perp}))\ar^-{}@/^1.2pc/[l]\ar_-{}@/_1.6pc/[l]}$$
}}
\end{thm}
\begin{proof} (1) We have projective cotorsion pairs $(dw^{\perp}(G^{\perp}), (dw^{\perp}(G^{\perp}))^\perp)$, $(ex^{\perp}(G^{\perp}), (ex^{\perp}(G^{\perp}))^\perp)$,
and $(dg\mathcal{P}, \mathcal{E})$. They satisfy that $dg\mathcal{P}\subseteq dw^{\perp}(G^{\perp})$ since $\mathcal{P}\subseteq{^{\perp}(G^{\perp})}$ and $dw^{\perp}(G^{\perp})\cap\mathcal{E}=ex^{\perp}(G^{\perp})$. The first recollement follows from [12, Theorem 4.7] and $K(dg\mathcal{P})\cong D(R)$.

(2) There are three projective cotorsion pairs $(dg^{\perp}(G^{\perp}), \widetilde{G^{\perp}})$, $(\widetilde{^{\perp}(G^{\perp})}, dg G^{\perp})$,
and $(dg\mathcal{P}, \mathcal{E})$. We know that $dg\mathcal{P}\subseteq dg^{\perp}(G^{\perp})$, $dg^{\perp}(G^{\perp})\cap\mathcal{E}=\widetilde{^{\perp}(G^{\perp})}$, and $K(dg\mathcal{P})\cong D(R)$. By [12, Theorem 4.7], the second recollement is obtained.

(3) Consider three projective cotorsion pairs $(^\perp(ex G^{\perp}), ex G^{\perp})$, $(^\perp(dwG^{\perp}), dw G^{\perp})$,
and $(dg\mathcal{P}, \mathcal{E})$. Since $ex G^{\perp}\subseteq\mathcal{E}$, we get $dg\mathcal{P}\subseteq {^\perp(ex G^{\perp})}$. By [8, Theorem 7.4.3], we have a Hovey triple $(^\perp(ex G^{\perp}), \mathcal{E},dw G^{\perp})$. Thus
$\mathcal{E}\cap{^\perp(ex G^{\perp})}={^\perp(dw G^{\perp})}$.
By [12, Theorem 4.7], we get the third recollement.

(4) $(dw^{\perp}(G^{\perp}), (dw^{\perp}(G^{\perp}))^\perp)$, $(dw\mathcal{P}, (dw\mathcal{P})^\perp)$,
and $(^\perp(dwG^{\perp}), dwG^{\perp})$ are three projective cotorsion pairs. By [12, Proposition 7.3], $^\perp(dw G^{\perp})\subseteq dw^{\perp}(G^{\perp})$.
However, $dw^{\perp}(G^{\perp})\cap dw G^{\perp}=dw(^{\perp}(G^{\perp})\cap G^{\perp})=dw\mathcal{P}$. Also, it immediately follows from [12, Theorem 4.7].

(5) By projective cotorsion pairs $(dw^{\perp}(G^{\perp}), (dw^{\perp}(G^{\perp}))^\perp)$, $(ex\mathcal{P}, (ex\mathcal{P})^\perp)$,
and $(^\perp(ex G^{\perp}), ex G^{\perp})$, we get $ex G^{\perp}\cap dw^{\perp}(G^{\perp})=(\mathcal{E}\cap dw G^{\perp})\cap dw^{\perp}(G^{\perp})=\mathcal{E}\cap dw(^{\perp}(G^{\perp})\cap G^{\perp})=ex\mathcal{P}$, and by [12, Proposition 7.3], $^\perp(exG^{\perp})\subseteq dw^{\perp}(G^{\perp})$.
By [12,Theorem 4.7], we obtain the last recollement.

This completes the proof.
\end{proof}

\begin{rem}\label{prop:2.4}{\rm{In this section, when $(\mathcal{L},\mathcal{A})$ is the level duality pair, one can obtain some specific results about $\mathcal{G}_{ac}\mathcal{P}$.
}}
\end{rem}

\begin{center}{\bf{Acknowledgments}}\end{center}
This work was supported by National Natural Science Foundation of China (No.11761060, 11901463), Science and Technology Project of Gansu Province (20JR5RA517),
Innovation Ability Enhancement Project of Gansu Higher Education Institutions (2019A-002) and Improvement of Young Teachers' Scientific Research Ability (NWNU-LKQN-18-30).

\begin{center}{\bf{References}}\end{center}

[1] V. Becerril, O. Mendoza, M.A. P$\acute{\mathrm{e}}$rez, V. Santiago, Frobenius pairs in abelian categories: correspondences with cotorsions pairs, exact model categories, and Auslander-Buchweitz contexts, \emph{J. Homotopy Relat. Struct.} \textbf{14} (2019) 1-50.

[2] A.A. Beilinson, J. Bernstein, P. Deligne, Faisceaux pervers, Ast\'{e}risque, \textbf{100} (1982) 5-171.

[3] D. Bravo, J. Gillespie, M. Hovey, The stable module category of a general ring (arXiv:1405.5768).

[4] T. B$\ddot{u}$hler, Exact categories, \emph{Expo. Math.} \textbf{28} (2010) 1-69.

[5] W.J. Chen, Z.K. Liu, X.Y. Yang, A new method to construct model structures from a cotorsion pair, \emph{Comm. Algebra} \textbf{47} (2019) 4420-4431.

[6] N.Q. Ding, Y.L. Li, L.X. Mao, Strongly Gorenstein flat modules, \emph{J. Aust. Math. Soc.} \textbf{86} (2009) 323-338.

[7] E.E. Enochs, O.M.G. Jenda, \emph{Relative Homological Algebra,} Berlin-New York: Walter de Gruyter. 2000

[8] E.E. Enochs, O.M.G. Jenda, \emph{Relative Homological Algebra}, vol. 2, de Gruyter Exp. Math., vol. 54, Walter de Gruyter. 2011

[9] E.E. Enochs, O.M.G. Jenda, Torrecillas, Gorenstein injective and projective modules, \emph{Math. Z.} \textbf{220} (1993) 611-633.

[10] J. Gillespie, Model structures on exact categories, \emph{J. Pure Appl. Algebra} \textbf{215} (2011) 2892-2902.

[11] J. Gillespie, How to construct a Hovey triple from two cotorsion pairs, \emph{Fund. Math.} \textbf{230} (2015) 281-289.

[12] J. Gillespie, Gorenstein complexes and recollements from cotorsion pairs, \emph{Adv. Math.} \textbf{291} (2016) 859-911.

[13] J. Gillespie, Duality pairs and stable module categores, \emph{J. Pure Appl. Algebra} \textbf{223} (2019) 3425-3435.

[14] R. G\"{o}bel, J. Trlifaj, \emph{Approximations and endomorphism Algebras of Modules,} Berlin-New York: Walter de Gruyter. 2012

[15] H. Holm, P. J$\mathrm{\phi}$rgensen, Cotorsion pairs induced by duality pairs, \emph{J. Commut. Algebra} \textbf{1} (2009) 621-633.

[16] M. Hovey, Cotorsion pairs, model category structures, and representation theory, \emph{Math. Z.} \textbf{241} (2002) 553-592.

[17] B. Keller, \emph{Derived categories and their uses,} In: Handbook of algebra, vol. 1, North-Holland, Amsterdam, pp. 671--701. 1996

[18] Z.P. Wang, G. Yang, R.M. Zhu, Gorenstein flat modules with respect to duality pairs, \emph{Comm. Algebra} \textbf{47} (2019) 4989-5006.

[19] X.Y. Yang, W.J. Chen, Relative homological dimensions and Tate cohomology of complexes with respect to cotorsion pairs, \emph{Comm. Algebra} \textbf{45} (2017) 2875-2888.

\end{document}